\documentclass[4pt]{article}
\usepackage[utf8]{inputenc}
\usepackage{geometry}
\usepackage{graphicx}
\usepackage[colorlinks,citecolor=red]{hyperref}
\usepackage{amsmath,amsthm}
\usepackage{amsfonts}
\usepackage{dsfont}
\usepackage{mathrsfs}
\usepackage{latexsym}
\usepackage{graphicx}
\usepackage{amssymb}
\usepackage{float}
\usepackage{epsfig}
\usepackage{epstopdf}
\usepackage{color,xcolor}
\usepackage{booktabs}
\usepackage{array}
\usepackage{paralist}
\usepackage{verbatim}
\usepackage{subfig}

\newtheorem{theorem}{Theorem}[section]
\newtheorem{lemma}[theorem]{Lemma}
\newtheorem{corollary}[theorem]{Corollary}

\usepackage{fancyhdr}
\usepackage{sectsty}
\allsectionsfont{\sffamily\mdseries\upshape}

\usepackage[nottoc,notlof,notlot]{tocbibind}
\usepackage[titles,subfigure]{tocloft}

\title{{\bf The $A_\alpha$-spectral radius of graphs with given degree sequence\thanks{Supported by NSFC (No.11531011), NSFXJ(No.2015KL019).}}}
\author{ Dan Li, Yuanyuan Chen, Jixiang Meng\footnote{Corresponding author. Email: mjxxju@sina.com}
\\
{\footnotesize College of Mathematics and System Sciences, Xinjiang University, Urumqi, Xinjiang, P.R.China}}

\date{}
\begin{document}
\maketitle

\begin{abstract}
Let $G$ be a graph with adjacency matrix $A(G)$, and let $D(G)$ be the diagonal matrix of the degrees of $G$. For any real $\alpha\in[0,1]$, write $A_\alpha(G)$ for the matrix $$A_\alpha(G)=\alpha D(G)+(1-\alpha)A(G).$$
This paper presents some extremal results about the spectral
radius $\rho(A_\alpha(G))$ of $A_\alpha(G)$ that generalize previous results about
$\rho(A_0(G))$ and $\rho(A_{\frac{1}{2}}(G))$. In this paper, we give some results on graph perturbation for $A_\alpha$-matrix with $\alpha\in [0,1)$. As applications, we characterize all extremal trees with the maximum $A_\alpha$-spectral radius in the set of all trees with prescribed degree sequence firstly. Furthermore, we characterize the unicyclic graphs that have the largest $A_\alpha$-spectral radius for a given unicycilc degree sequence.


\end{abstract}

\section{Introduction}
All graphs considered in this paper are simple and undirected. Let $G$ be a graph with adjacency matrix $A(G)$, and let $D(G)$ be the diagonal matrix of the degrees of $G$. For any real $\alpha\in[0,1]$, Nikiforov\cite{V. Nikiforov2} defined the matrix $A_\alpha(G)$ as $$A_\alpha(G)=\alpha D(G)+(1-\alpha)A(G).$$It is clear that $A_\alpha(G)$ is adjacency matrix if $\alpha=0$, and $A_\alpha(G)$ is essentially equivalent to signless Laplacain matrix if $\alpha=\frac{1}{2}$. We denote the eigenvalues of $A_\alpha(G)$ by $\lambda_1(A_\alpha(G))\geq\lambda_2(A_\alpha(G))\geq\cdots\geq\lambda_n(A_\alpha(G))$. Write $\rho(A_\alpha(G))$ for the spectral radius of $A_\alpha(G)$ and call it $A_\alpha$-spectral radius of $G$. Let $d_G(v)$, or $d(v)$ for short, be the degree of the vertex $v$ in $G$. A nonincreasing sequence of nonnegative integers $\pi=(d_0,d_1,\ldots,d_{n-1})$ is called {\it graphic} if there exists a simple graph $G$ with order $n$ having $\pi$ as its vertex degree sequence. Let $N_G(v)$ or $N(v)$ denote the neighbor set of vertex $v$ of $G$.

One of the central issues in spectral extremal graph theory is: For a graph matrix, determine the maximization or minimization of spectral invariants over various families of graphs. Nikiforov et al.\cite{V. Nikiforov4} showed that $P_n$, the path of order $n$, has minimal $A_\alpha$-spectral radius among all connected graphs of order $n$. Xue, Lin et al.\cite{Xue Jie1} determined the unique graph with maximum $A_\alpha$-spectral radius among all connected graphs with diameter $d$, and determined the unique graph with minimum $A_\alpha$-spectral radius among all connected graphs with given clique number. Nikiforov\cite{V. Nikiforov5} presented some extremal results about the $A_\alpha-$spectral
radius that generalize previous results about $\rho(A_0(G))$ and $\rho(A_{\frac{1}{2}}(G))$.

Zhang \cite{Zhang xiao dong3,Zhang xiao dong2} posed the following problem about $\rho(A_0(G))$ and $\rho(A_{\frac{1}{2}}(G))$. We generalize it to $0\leq \alpha<1.$

\vspace*{3mm}
\noindent{\bf Problem.}\label{p1} For a given graphic degree sequence $\pi$, let $0\leq \alpha<1$ and
$$\mathcal{G}_{\pi}=\{G\mid \mbox{$G$ is connected with $\pi$ as its degree sequence}\}.$$
Find the upper (lower) bounds for the $A_\alpha$-spectral radius of all graphs $G$ in $\mathcal{G}_{\pi}$ and characterize all extremal graphs which attain the upper (lower) bounds.

Biyiko\u{g}lu et al.\cite{T. Biyikolu} determined the unique tree with maximum
$A_0$-spectral radius in the set of all trees with prescribed degree sequence. Belardo et al.\cite{B. Francescoelardo} determined the (unique) graphs with the largest $A_0$-spectral radius in the set of all unicyclic graphs with prescribed degree sequence. Zhang \cite{Zhang xiao dong3} determined the unique tree with maximum
$A_{\frac{1}{2}}$-spectral radius in the set of all trees with prescribed degree sequence. Zhang \cite{Zhang xiao dong2} characterized the unicyclic graphs that have the largest $A_{\frac{1}{2}}$-spectral radius in the set of all unicyclic graphs for a given unicycilc graphic degree sequence. The main goal of this paper is to extend their results for all $\alpha\in[0,1).$

To generalize these results, we first consider some results on graph perturbation for $A_\alpha$-matrix in section 2. Using these basic tools, we first determine the unique tree with maximum $A_\alpha$-spectral radius in the set of all trees with prescribed degree sequence in section 3.

\begin{theorem}\label{Theorem3.3} For a given nonincreasing tree degree sequence $\pi=(d_0,d_1,\ldots,d_{n-1})$, if $0\leq \alpha<1$, then $\mathcal{T}^*_{\pi}$ (see in section 3) has largest $A_\alpha$-spectral radius in the class of all trees with degree sequence $\pi$.
\end{theorem}

In section 4, we characterize the unicyclic graphs with maximum $A_\alpha$-spectral radius in the set of all unicyclic graphs with prescribed degree sequence.

\begin{theorem}\label{Theorem4.7} Let $\pi=(d_0,d_1,\ldots,d_{}n-1)$ be a positive nonincreasing integer sequence with $\sum\limits_{i=0}^{n-1}d_i=2n$. If $0\leq \alpha<1$, then $\rho(A_{\alpha}(\mathcal{U}^*_{\pi}))= \max\{\rho(A_{\alpha}(H))\mid \mbox{H is unicyclic and} ~H\in \mathcal{G}_{\pi}\}$
, where $\mathcal{U}^*_{\pi}$ is shown in section 4.
\end{theorem}

\section{Basic tools}
\begin{lemma}\label{Lemma2.1}\cite{V. Nikiforov5,Xue Jie1} Let $G$ be a connected graph with $\alpha\in[0,1)$. For $u,v\in V(G)$, suppose $N\subseteq N(v)\setminus (N(u)\cup \{u\})$. Let $G'=G-\{vw:w\in N\}+\{uw:w\in N\}$. Let $X$ be a unit eigenvector of $A_{\alpha}(G)$ corresponding to $\rho(A_{\alpha}(G))$. If $N\neq \emptyset$ and $x_u\geq x_v$, then $\rho(A_{\alpha}(G'))>\rho(A_{\alpha}(G))$.
\end{lemma}

\begin{lemma}\label{Lemma2.2} Let $G\in\mathcal{G}_{\pi}$ be a connected graph with $\alpha\in[0,1)$. Let $d_G(u)-d_G(v)=k>0$ and $X$ be a unit eigenvector of $A_{\alpha}(G)$ corresponding to $\rho(A_{\alpha}(G))$. If $x_v\geq x_u$, then there exists a connected graph $G'\in\mathcal{G}_{\pi}$ such that $\rho(A_{\alpha}(G'))>\rho(A_{\alpha}(G))$.
\end{lemma}
\begin{proof}From $d_G(u)-d_G(v)=k>0$, we know that $d_G(u)>d_G(v)$, then $N(u)\setminus (N(v)\cup \{v\})\neq \emptyset$. Combining with $x_v\geq x_u$, the assertion holds by Lemma \ref{Lemma2.1}.
\end{proof}

\begin{lemma}\label{Lemma2.3} Let $G\in\mathcal{G}_{\pi}$ be a connected graph with $\alpha\in[0,1)$. Let $X$ be a unit eigenvector of $A_{\alpha}(G)$ corresponding to $\rho(A_{\alpha}(G))$. Assume that $v_1u_1, v_2u_2\in E(G)$ and $v_1 v_2,u_1 u_2\notin E(G)$. Let $G'$ be a new graph obtained from $G$ by deleting edges $v_1u_1, v_2u_2$ and adding edges $v_1 v_2,u_1 u_2$. If $x_{v_1}\geq x_{u_2}$ and $x_{v_2}\geq x_{u_1}$, then $\rho(A_{\alpha}(G'))\geq\rho(A_{\alpha}(G))$. Furthermore, if one of the two inequalities is strict, then $\rho(A_{\alpha}(G'))>\rho(A_{\alpha}(G))$.
\end{lemma}
\begin{proof}Let $X$ be a unit eigenvector of $A_{\alpha}(G)$ corresponding to $\rho(A_{\alpha}(G))$. Then
\begin{eqnarray*}
\rho(A_{\alpha}(G'))-\rho(A_{\alpha}(G))&\geq&X^T(A_{\alpha}(G')-A_{\alpha}(G))X\\
&=& 2(1-\alpha)(x_{v_1}-x_{u_2})(x_{v_2}-x_{u_1})\\
&\geq& 0.
\end{eqnarray*}
If $\rho(A_{\alpha}(G'))=\rho(A_{\alpha}(G))$, then $X$ is also a unit eigenvector of $A_{\alpha}(G')$ corresponding to $\rho(A_{\alpha}(G'))$. And
$$(\rho(A_{\alpha}(G'))-\rho(A_{\alpha}(G)))x_{v_1}=(1-\alpha)(x_{v_2}-x_{u_1})=0,$$
thus $x_{v_2}=x_{u_1}$. Similarly, we get $x_{v_1}=x_{u_2}$.
\end{proof}

\begin{lemma}\label{Lemma2.4} Let $G\in\mathcal{G}_{\pi}$ be a connected graph with $\alpha\in[0,1)$. Let $X$ be a unit eigenvector of $A_{\alpha}(G)$ corresponding to $\rho(A_{\alpha}(G))$. If there exist three vertices $u,v,w\in V(G)$ such that $uv\in E(G)$, $uw\notin E(G)$ and $x_v<x_w\leq x_u$, $x_u\geq x_t$ for any vertex $t\in N(w)$. Then there exists a connected graph $G'\in\mathcal{G}_{\pi}$ such that $\rho(A_{\alpha}(G'))>\rho(A_{\alpha}(G))$.
\end{lemma}
\begin{proof}Let $X$ be a unit eigenvector of $A_{\alpha}(G)$ corresponding to $\rho(A_{\alpha}(G))$. Then the following claims must hold.

\noindent{\bf Claim.} There exists a vertex $t\in N(w)$ such that $tv\notin E(G)$ and $t\neq v$

Otherwise, suppose that $N(w)\setminus \{v\}\subset N(v)$, then $N(w)\setminus \{v\}\subset N(v)\setminus \{u\}$ since $uw\notin E(G)$, and then $d_G(w)<d_G(v)$. By $A_{\alpha}(G)X=\rho(A_{\alpha}(G))X$, we obtain that
$$\rho(A_{\alpha}(G))x_w=\alpha d_G(w)x_w+(1-\alpha)\sum\limits_{j\in N(w)}x_j,$$
$$\rho(A_{\alpha}(G))x_v=\alpha d_G(v)x_v+(1-\alpha)\sum\limits_{j\in N(v)}x_j,$$
then
\begin{eqnarray*}
(\rho(A_{\alpha}(G))-\alpha d_G(w))x_w&=&(1-\alpha)\sum\limits_{j\in N(w)}x_j\leq (1-\alpha)(x_v+\sum\limits_{j\in N(w)\backslash \{v\}}x_j)\\
&<& (1-\alpha)(x_v+\sum\limits_{j\in N(v)\backslash \{u\}}x_j)\\
&<& (1-\alpha)(x_u+\sum\limits_{j\in N(v)\backslash \{u\}}x_j)\\
&=& (1-\alpha)\sum\limits_{j\in N(v)}x_j\\
&=& (\rho(A_{\alpha}(G))-\alpha d_G(v))x_v.
\end{eqnarray*}
Since $x_v<x_w$, $\rho(A_{\alpha}(G))-\alpha d_G(w)<\rho(A_{\alpha}(G))-\alpha d_G(v)$, i.e., $d_G(w)>d_G(v)$, a contradiction.

\noindent{\bf Claim.} $|N(w)|\geq2$.

We try to prove the claim by contradiction. Note that $|N(w)|\geq1$ by Claim 1. So suppose that $N(w)=\{p\}$ and $pv\notin E(G)$, $p\neq v$,
then
$$\rho(A_{\alpha}(G))x_w=\alpha x_w+(1-\alpha)x_p,$$
$$\rho(A_{\alpha}(G))x_v=\alpha d_G(v)x_v+(1-\alpha)\sum\limits_{j\in N(v)}x_j,$$
since $x_w>x_v$,
$$(\rho(A_{\alpha}(G))-\alpha)x_w>(\rho(A_{\alpha}(G))-\alpha)x_v\geq(\rho(A_{\alpha}(G))-\alpha d_G(v))x_v,$$
that is,
$$(1-\alpha)x_p>(1-\alpha)x_u+(1-\alpha)\sum\limits_{j\in N(v)\setminus \{u\}}x_j\geq(1-\alpha)x_u,$$
thus $x_p>x_u$, a contradiction.

\noindent{\bf Claim.} There exists a connected graph $G'\in\mathcal{G}_{\pi}$ such that $\rho(A_{\alpha}(G'))>\rho(A_{\alpha}(G))$.

Since $G$ is connected and $uw\notin E(G)$, there exist a path $P_{uw}$ between $u$ and $w$. Let $sw\in E(P_{uw})$.

\noindent{\bf Case 1.} $uv\notin E(P_{uw})$.

\noindent{\bf Subcase 1.1.} $vs\notin E(G)$.
For the edges $uv, ws\in E(G)$, combining with $x_u\geq x_s$ and $x_w>x_v$, then we can construct a new connected graph $G'\in\mathcal{G}_{\pi}$ from $G$ by deleting edges $uv, ws$ and adding edges $uw,vs$ such that $\rho(A_{\alpha}(G'))>\rho(A_{\alpha}(G))$ by Lemma \ref{Lemma2.3}.

\noindent{\bf Subcase 1.2.} $vs\in E(G)$.
By Claim 1, there exists a vertex $p\in N(w)$ such that $pv\notin E(G)$ and $p\neq v$. Now we consider the edges $uv, wp\in E(G)$, combining with $x_u\geq x_p$ and $x_w>x_v$, then we can construct a new connected graph $G'\in\mathcal{G}_{\pi}$ from $G$ by deleting edges $uv, wp$ and adding edges $uw,vp$ such that $\rho(A_{\alpha}(G'))>\rho(A_{\alpha}(G))$ by Lemma \ref{Lemma2.3}.

\noindent{\bf Case 2.} $uv\in E(P_{uw})$.

\noindent{\bf Subcase 2.1.} $vs\notin E(G)$ and $vt\in E(G)$ for any vertex $t\in N(w)\backslash \{s\}$.
Then we consider the edges $uv, ws\in E(G)$, combining with $x_u\geq x_s$ and $x_w>x_v$, then we can construct a new connected graph $G'\in\mathcal{G}_{\pi}$ from $G$ by deleting edges $uv, ws$ and adding edges $uw,vs$ such that $\rho(A_{\alpha}(G'))>\rho(A_{\alpha}(G))$ by Lemma \ref{Lemma2.3}.

\noindent{\bf Subcase 2.2.} There exists a vertex $p\in N(w)\backslash \{s\}$ such that $pv\notin E(G)$.
Now we consider the edges $uv, wp\in E(G)$, combining with $x_u\geq x_p$ and $x_w>x_v$, then we can construct a new connected graph $G'\in\mathcal{G}_{\pi}$ from $G$ by deleting edges $uv, wp$ and adding edges $uw,vp$ such that $\rho(A_{\alpha}(G'))>\rho(A_{\alpha}(G))$ by Lemma \ref{Lemma2.3}.
\end{proof}

\begin{lemma}\label{Lemma2.5} Let $G\in\mathcal{G}_{\pi}$ be a connected graph with $\alpha\in[0,1)$ and $V(G)=\{v_0,v_1,\ldots,v_{n-1}\}$. Let $\rho(A_{\alpha}(G))= \max\{\rho(A_{\alpha}(H))|H\in \mathcal{G}_{\pi}\}$ and $X$ be a unit eigenvector of $A_{\alpha}(G)$ corresponding to $\rho(A_{\alpha}(G))$. Then the following assertions hold.
\begin{description}
  \item[(1)] If $x_{v_i}\geq x_{v_j}$, then $d_G(v_i)\geq d_G(v_j)$ for $i<j$;
  \item[(2)] If $x_{v_i}= x_{v_j}$, then $d_G(v_i)= d_G(v_j)$.
\end{description}
\end{lemma}
\begin{proof}(1). If $d_G(v_i)< d_G(v_j)$ for $i<j$, combining with $x_{v_i}\geq x_{v_j}$, then there exists a connected graph $G'\in\mathcal{G}_{\pi}$ such that $\rho(A_{\alpha}(G'))>\rho(A_{\alpha}(G))$ by Lemma \ref{Lemma2.2}, a contradiction. Thus, $d_G(v_0)\geq d_G(v_1)\geq\cdots\geq d_G(v_{n-1})$.

\noindent(2). If $x_{v_i}= x_{v_j}$, then we have $d_G(v_i)\leq d_G(v_j)$ by the same argument as (1). So $d_G(v_i)= d_G(v_j)$.
\end{proof}

For a graph with a root $v_0$, we call the distance the {\it height} $h(v)=\mbox{dis}(v,v_0)$ of a vertex $v.$

\noindent{\bf Definition 1.} Let $G=(V,E)$ be a graph with root $v_0$. A well-ordering $\prec$ of the vertices is called a breadth-first-search ordering (BFS-ordering for short) if the following hold for all vertices $u,v\in V$:
\begin{description}
  \item[(1)] $u\prec v$ implies $h(u)\leq h(v)$;
  \item[(2)] $u\prec v$ implies $d_G(u)\geq d_G(v)$;
  \item[(3)] let $uv\in E(G)$, $xy\in E(G)$, $uy\notin E(G)$, $xv\notin E(G)$ with $h(u)=h(x)=h(v)-1=h(y)-1$. If $u\prec x$, then $v\prec y$.
\end{description}
We call a graph that has a BFS-ordering of its vertices a BFS-graph.

\begin{lemma}\label{Lemma2.6} Let $G\in\mathcal{G}_{\pi}$ be a connected graph with $\alpha\in[0,1)$ and $V(G)=\{v_0,v_1,\ldots,v_{n-1}\}$. Let $\rho(A_{\alpha}(G))= \max\{\rho(A_{\alpha}(H))|H\in \mathcal{G}_{\pi}\}$ and $X$ be a unit eigenvector of $A_{\alpha}(G)$ corresponding to $\rho(A_{\alpha}(G))$. Then there exists a numeration of the vertices of $G$ such that $x_{v_0}\geq x_{v_1}\geq \cdots\geq x_{v_{n-1}}$ and $h(v_0)\leq h(v_1)\leq\cdots\leq h(v_{n-1})$.
\end{lemma}
\begin{proof}At first, we can find a numeration of the vertices of $G$ such that $v_0\prec v_1\prec\cdots\prec v_{n-1}$ and $x_{v_0}\geq x_{v_1}\geq \cdots\geq x_{v_{n-1}}$. Next, we just need to show that $h(v_i)\leq h(v_{i+1})$ by induction. If $i=0$, then $h(v_0)=0<h(v_1)$ clearly. Assume that $h(v_i)\leq h(v_{i+1})$ for $i\leq k-1$. Now, we try to prove $h(v_k)\leq h(v_{k+1})$. Let $V=V_1\cup V_2$, where $V_1=\{v_0,v_1,\ldots,v_k\}$, $V_2=\{v_{k+1},\ldots,v_{n-1}\}$, without loss of generality, assume that $x_{v_k}>x_{v_{k+1}}$. Since $G$ is connected, there exists an edge between $V_1$ and $V_2$, and let $v_p$ be the first vertex of $V_1$ such that $v_pv_r\in E(G)$, where $v_r\in N(V_2)$.

\noindent{\bf Case 1.} $h(v_p)\geq h(v_{k})-1$.
Assume that $P_{v_0v_{k+1}}$ is the shortest path between $v_0$ and $v_{k+1}$ and $v_q\in V_1$ is that last one of $P_{v_0v_{k+1}}$, then $h(v_p)\leq h(v_q)$. And then $h(v_{k+1})\geq h(v_q)+1\geq h(v_p)+1\geq h(v_{k})$.

\noindent{\bf Case 2.} $h(v_p)< h(v_{k})-1$, which implies that $v_pv_k\notin E(G)$.
Note that $x_{v_r}\leq x_{v_{k+1}}<x_{v_k}\leq x_{v_p}$. And we claim that $x_{v_s}\leq x_{v_p}$ for any vertex $v_s\in N(v_k)$. If $v_s\in V_2$, then the claim holds obviously. If $v_s\in V_1$ and $x_{v_s}> x_{v_p}$, then $0\leq s<p\leq k$ and $h(v_s)\leq h(v_p)$. Thus, $h(v_k)\leq h(v_s)+1\leq h(v_p)<h(v_k)$, a contradiction.

Now, we consider the vertices $v_p$, $v_k$ and $v_r$. Note that $x_{v_r}<x_{v_k}\leq x_{v_p}$ and $x_{v_s}\leq x_{v_p}$ for any vertex $x_{v_s}\in N(v_k)$. Then by Lemma \ref{Lemma2.4}, there exists a connected graph $G'\in\mathcal{G}_{\pi}$ such that $\rho(A_{\alpha}(G'))>\rho(A_{\alpha}(G))$, a contradiction.
\end{proof}

\begin{theorem}\label{Theorem2.7} Let $G\in\mathcal{G}_{\pi}$ be a connected graph with $\alpha\in[0,1)$. If $\rho(A_{\alpha}(G))= \max\{\rho(A_{\alpha}(H))\mid H\in\mathcal{G}_{\pi}\}$, then $G$ has a BFS-ordering.
\end{theorem}
\begin{proof}Let $X$ be a unit eigenvector of $A_{\alpha}(G)$ corresponding to $\rho(A_{\alpha}(G))$. By Lemmas \ref{Lemma2.5} and \ref{Lemma2.6},
there exists a well-ordering $v_0\prec v_1\prec\cdots\prec v_{n-1}$ such that
$$x_{v_0}\geq x_{v_1}\geq \cdots\geq x_{v_{n-1}},$$
$$d_G(v_0)\geq d_G(v_1)\geq \cdots\geq d_G(v_{n-1}),$$
and
$$h(v_0)\leq h(v_1)\leq \cdots\leq h(v_{n-1}).$$
Let $v_iv_s,v_jv_t\in E(G)$ and $v_iv_t,v_jv_s\notin E(G)$ with $h(v_i)=h(v_j)=h(v_t)-1=h(v_s)-1$ and $v_i\prec v_j$, we just need to show that $v_s\prec v_t$. Otherwise, without loss of generality, assume that $x_{v_s}<x_{v_t}$. Then we can construct a new connected graph $G'\in\mathcal{G}_{\pi}$ from $G$ by deleting edges $v_iv_s, v_jv_t$ and adding edges $v_iv_t,v_jv_s$ such that $\rho(A_{\alpha}(G'))>\rho(A_{\alpha}(G))$ by Lemma\ref{Lemma2.3}, a contradiction.
\end{proof}

From the proof of Theorem \ref{Theorem2.7}, the following corollary is obtained easily.
\begin{corollary}\label{Corollary2.7} Let $G\in\mathcal{G}_{\pi}$ be a connected graph with $\alpha\in[0,1)$ and $\rho(A_{\alpha}(G))= \max\{\rho(A_{\alpha}(H))\mid H\in \mathcal{G}_{\pi}\}$. Then $G$ has a BFS-ordering consistent with the unit eigenvector $X$ corresponding to $\rho(A_{\alpha}(G))$ in such a way that $u\prec v$ implies $x_u\geq x_v$.
\end{corollary}

\begin{lemma}\label{Lemma2.8}\cite{Berman A} Let $M$ be a nonnegative irreducible symmetric matrix with spectral radius $\rho(M)$. If there exists a positive vector $Y>0$ and a positive real $\beta$ such that $MY<\beta Y$, then $\rho(M)<\beta$.
\end{lemma}

\begin{lemma}\label{Lemma2.9}\cite{V. Nikiforov2,V. Nikiforov4} Let $G$ be a graph with maximal degree $\vartriangle$ and $\alpha\in[0,1)$. Then
$$\rho(A(G))\leq\rho(A_\alpha(G))\leq \vartriangle.$$
If the left equality holds, then $G$ has a $\rho(A(G))$-regular component. If the right equality holds, then $G$ is regular.
\end{lemma}
\setlength{\unitlength}{0.6pt}
\begin{center}
\begin{picture}(209,79)
\put(0,79){\circle*{4}}
\put(31,56){\circle*{4}}
\qbezier(0,79)(15,68)(31,56)
\put(1,33){\circle*{4}}
\qbezier(1,33)(16,45)(31,56)
\put(71,56){\circle*{4}}
\qbezier(31,56)(51,56)(71,56)
\put(131,56){\circle*{4}}
\multiput(71,56)(10.00,0.00){6}{\qbezier(0,0)(0,0)(7.50,0.00)}
\put(173,56){\circle*{4}}
\qbezier(131,56)(152,56)(173,56)
\put(208,79){\circle*{4}}
\qbezier(173,56)(190,68)(208,79)
\put(209,36){\circle*{4}}
\qbezier(173,56)(191,46)(209,36)
\end{picture}
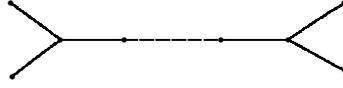
\centerline{Fig. The graph $H$.}
\end{center}

\noindent{\bf Definition 1.} Let $G$ be a simple graph. An {\it internal path} of $G$ is a path $P$ (or cycle) with vertices $v_1,v_2,\ldots,v_k$ (or $v_1=v_k$) such that $d_G(v_1)\geq 3$, $d_G(v_k)\geq 3$ and $d_G(v_2)=\cdots=d_G(v_{k-1})=2$.
\begin{lemma}\label{Lemma2.10} Let $G$ be a connected graph with $\alpha\in[0,1)$ and $uv$ be an edge on the internal path of $G$. If $G_{uv}$ is obtained from $G$ by subdivision of edge $uv$ into edges $uw$ and $wv$, then $\rho(A_\alpha(G_{uv}))<\rho(A_\alpha(G))$.
\end{lemma}
\begin{proof}Let $P=u_0u_1\ldots u_{k+1}$ be the internal path of $G$ and $X$ be a unit eigenvector of $A_{\alpha}(G)$ corresponding to $\rho(A_{\alpha}(G))$ in which $x_i$ is corresponding to $u_i$. Without loss of generality, assume that $x_0\leq x_{k+1}$. Let $t$ be the smallest index such that $x_t=\min\limits_{0\leq i\leq k+1}x_i$, so $t<k+1$. Without loss of generality, assume that $u=u_t$ and $v=u_{t+1}$. Note that $H\subset G$ and $\rho(A(H))=2$. Since $G$ is not a regular graph, $\rho(A_{\alpha}(G))>2$ by Lemma \ref{Lemma2.9}.

\noindent{\bf Case 1.} $t>0$. We can construct a positive vector $Y$ as follows.
$$y_p=\left\{
    \begin{array}{ll}
      x_p, & \mbox{if} ~~p\neq w; \\[3mm]
      x_u, & \mbox{if} ~~p=w.
    \end{array}
  \right.$$
If $x_{t-1}>x_u$, then
\begin{eqnarray*}
(A_{\alpha}(G_{uv})Y)_u&=&2\alpha x_u+(1-\alpha)(x_{t-1}+x_u)\\
&\leq& 2\alpha x_u+(1-\alpha)(x_{t-1}+x_v)\\
&=& \rho(A_{\alpha}(G))x_u\\
&=& \rho(A_{\alpha}(G))y_u.
\end{eqnarray*}
and
\begin{eqnarray*}
(A_{\alpha}(G_{uv})Y)_w&=&2\alpha x_u+(1-\alpha)(x_u+x_v)\\
&<& 2\alpha x_u+(1-\alpha)(x_{t-1}+x_v)\\
&=& \rho(A_{\alpha}(G))x_u\\
&=& \rho(A_{\alpha}(G))y_w.
\end{eqnarray*}
If $x_{t-1}=x_u$, then
$$(A_{\alpha}(G_{uv})Y)_u=2x_u<\rho(A_{\alpha}(G))x_u=\rho(A_{\alpha}(G))y_u$$and$$(A_{\alpha}(G_{uv})Y)_w=(A_{\alpha}(G_{uv})Y)_w.$$
Therefore, $A_{\alpha}(G_{uv})Y<\rho(A_{\alpha}(G))Y$ and $\rho(A_\alpha(G_{uv}))<\rho(A_\alpha(G))$ by Lemma \ref{Lemma2.8}.

\noindent{\bf Case 2.} $t=0$, i.e., $d_G(u)\geq3$. Let $S=\sum\limits_{j\in N_G(u)\setminus\{v\}}x_j$.

\noindent{\bf Subcase 2.1.} $\rho(A_\alpha(G))x_u>2\alpha x_u+(1-\alpha)x_u+(1-\alpha)x_v$.

We can construct a positive vector $Y$ as follows.
$$y_p=\left\{
    \begin{array}{ll}
       x_p, & \mbox{if} ~~p\neq w; \\[3mm]
       x_u, & \mbox{if} ~~p=w.
    \end{array}
  \right.$$
Then
\begin{eqnarray*}
(A_{\alpha}(G_{uv})Y)_u&=&\alpha d_G(u)x_u+(1-\alpha)(S+x_u)\\
&\leq& 2\alpha x_u+(1-\alpha)(S+x_v)\\
&=& \rho(A_{\alpha}(G))x_u\\
&=& \rho(A_{\alpha}(G))y_u.
\end{eqnarray*}
\begin{eqnarray*}
(A_{\alpha}(G_{uv})Y)_w&=&2\alpha x_u+(1-\alpha)(x_u+x_v)\\
&<& \rho(A_{\alpha}(G))x_u\\
&=& \rho(A_{\alpha}(G))y_w.
\end{eqnarray*}
Therefore, $A_{\alpha}(G_{uv})Y<\rho(A_{\alpha}(G))Y$ and $\rho(A_\alpha(G_{uv}))<\rho(A_\alpha(G))$ by Lemma \ref{Lemma2.8}.

\noindent{\bf Subcase 2.2.} $\rho(A_\alpha(G))x_u\leq2\alpha x_u+(1-\alpha)x_u+(1-\alpha)x_v$.

We can construct a positive vector $Y$ as follows.
$$y_p=\left\{
    \begin{array}{ll}
      x_p, & \mbox{if} ~~p\neq u,w; \\[3mm]
      x_u, & \mbox{if} ~~p=w; \\[3mm]
      \frac{1}{1-\alpha}[(\rho(A_\alpha(G))-2\alpha)x_u-(1-\alpha)x_v] & \mbox{if} ~~p=u.
    \end{array}
  \right.$$
Note that $$\rho(A_{\alpha}(G))x_u=\alpha d_G(u)x_u+(1-\alpha)(S+x_v),$$
thus $$y_u=\frac{\alpha}{1-\alpha}(d_G(u)-2)x_u+s>0.$$
Obviously, by $\rho(A_{\alpha}(G))>2$ we have
$$2x_u-x_v< y_u\leq x_u,$$
that is,
$$y_u\leq x_u<x_v.$$
By $y_u\leq x_u$, we observe that
$$S\leq \frac{1+\alpha-\alpha d_G(u)}{1-\alpha}x_u.$$
Then
\begin{eqnarray*}
(A_{\alpha}(G_{uv})Y)_w&=&2\alpha x_u+(1-\alpha)(y_u+x_v)\\
&=& 2\alpha x_u+(\rho(A_{\alpha}(G))-2\alpha)x_u-(1-\alpha)x_v+(1-\alpha)x_v\\
&=& \rho(A_{\alpha}(G))x_u\\
&=& \rho(A_{\alpha}(G))y_w.
\end{eqnarray*}
For any vertex $z\in N_G(u)\setminus \{v\}$,
\begin{eqnarray*}
(A_{\alpha}(G_{uv})Y)_z&=&\alpha d_G(z)x_z+(1-\alpha)(\sum\limits_{j\in N_G(z)\setminus \{u\}}x_j+y_u)\\
&\leq& \alpha d_G(z)x_z+(1-\alpha)(\sum\limits_{j\in N(z)\setminus \{u\}}x_j+x_u)\\
&=& \rho(A_{\alpha}(G))x_z\\
&=& \rho(A_{\alpha}(G))y_z.
\end{eqnarray*}
Next, we only need to check it at the vertex $u$. Note that
$$\rho(A_{\alpha}(G))S~\geq \alpha S+(1-\alpha)(d_G(u)-1)x_u.$$
Then
\begin{eqnarray*}
(A_{\alpha}(G_{uv})Y)_u&=&\alpha d_G(u)y_u+(1-\alpha)(S+x_u)\\
&=& \frac{\alpha^2}{1-\alpha}d_G(u)(d_G(u)-2)x_u+\alpha(d_G(u)-1)S+S+(1-\alpha)x_u.
\end{eqnarray*}
\begin{eqnarray*}
\rho(A_{\alpha}(G))y_u&=&\rho(A_{\alpha}(G))[\frac{\alpha}{1-\alpha}(d_G(u)-2)x_u+S]\\
&=& \frac{\alpha}{1-\alpha}(d_G(u)-2)\rho(A_{\alpha}(G))x_u+\rho(A_{\alpha}(G))S\\
&\geq& \frac{\alpha^2}{1-\alpha}d_G(u)(d_G(u)-2)x_u+\alpha(d_G(u)-1)S\\
&+& \alpha(d_G(u)-2)x_v+(1-\alpha)(d_G(u)-1)x_u.
\end{eqnarray*}
And combining with $x_u<x_v$, we get
\begin{eqnarray*}
\rho(A_{\alpha}(G))y_u-(A_{\alpha}(G_{uv})Y)_u&\geq&\alpha(d_G(u)-2)x_v+(1-\alpha)(d_G(u)-2)x_u-S\\
&\geq& \alpha(d_G(u)-2)x_v+(1-\alpha)(d_G(u)-2)x_u-\frac{1+\alpha-\alpha d_G(u)}{1-\alpha}x_u\\
&>& \alpha(d_G(u)-2)x_u+(1-\alpha)(d_G(u)-2)x_u-\frac{1+\alpha-\alpha d_G(u)}{1-\alpha}x_u\\
&=& \frac{d_G(u)-3+\alpha}{1-\alpha}x_u\\
&\geq& 0.
\end{eqnarray*}
Therefore, $A_{\alpha}(G_{uv})Y<\rho(A_{\alpha}(G))Y$ and $\rho(A_\alpha(G_{uv}))<\rho(A_\alpha(G))$ by Lemma \ref{Lemma2.8}.
\end{proof}

\section{The $A_{\alpha}$-spectral radius of trees}
The graph $\mathcal{T}^*_{\pi}$ has been introduced by Zhang \cite{Zhang xiao dong3}, let's go over it.
For a given nonincreasing degree sequence $\pi=(d_0,d_1,\ldots,d_{n-1})$ of a tree with $n\geq3$, the $\mathcal{T}^*_{\pi}$ can be construct as follows. Assume that $d_m>1$ and $d_{m+1}=\cdots=d_{n-1}=1$ for $0\leq m<n-1$. Put $s_0=0$, select a vertex $v_{01}$ as a root and begin with $v_{01}$ in layer 0. Put $s_1=d_0$ and select $s_1$ vertices $\{v_{11},\ldots,v_{1s_1}\}$ in layer 1 such that they are adjacent to $v_{01}$. Thus $d(v_{01})=d_0=s_1$. We continue to construct all other layer by recursion. In general, put $s_t=d_{s_0+s_1+\cdots+s_{t-2}+1}+\cdots+d_{s_0+s_1+\cdots+s_{t-2}+s_{t-1}}-s_{t-1}$ for $t\geq2$ and assume that all vertices in layer $t$ have been constructed and are denoted by $\{v_{t1},\ldots,v_{ts_t}\}$ with $d(v_{t-1,1})=d_{s_0+s_1+\cdots+s_{t-2}+1}$,\ldots,$d(v_{t-1,s_{t-1}})=d_{s_0+s_1+\cdots+s_{t-2}+s_{t-1}}.$
Now using the induction hypothesis, we construct all vertices in layer $t+1$. Put $s_{t+1}=d_{s_0+\cdots+s_{t-1}+1}+\cdots+d_{s_0+s_1+\cdots+s_{t-2}+s_{t-1}+s_t}-s_t$. Select $s_{t+1}$ vertices $\{v_{t+1,1},\ldots,v_{t+1,s_{t+1}}\}$ in layer $t+1$ such that $v_{t+1,i}$ is adjacent to $v_{tr}$ for $r=1$ and $1\leq i\leq d_{s_0+\cdots+s_{t-1}+1}-1$ and for $2\leq r\leq s_t$ and $d_{s_0+\cdots+s_{t-1}+1}+d_{s_0+\cdots+s_{t-1}+2}+\cdots+d_{s_0+\cdots+s_{t-1}+r-1}-r+2\leq i\leq d_{s_0+\cdots+s_{t-1}+1}+d_{s_0+\cdots+s_{t-1}+2}+\cdots+d_{s_0+\cdots+s_{t-1}+r}-r$. Thus $d(v_{tr})=d_{s_0+\cdots+s_{t-1}+r}$ for $1\leq r\leq s_t$. Assume that $m=s_0+\cdots+s_{p-1}+q$. Put $s_{p+1}=d_{s_0+\cdots+s_{p-1}+1}+\cdots+d_{s_0+\cdots+s_{p-1}+q}-q$ and select $s_{p+1}$ vertices $\{v_{p+1,1},\ldots,v_{p+1,s_{p+1}}\}$ in layer $p+1$ such that $v_{p+1,i}$ is adjacent to $v_{pr}$ for $1\leq r\leq q$ and $d_{s_0+\cdots+s_{p-1}+1}+d_{s_0+\cdots+s_{p-1}+2}+\cdots+d_{s_0+\cdots+s_{p-1}+r-1}-r+2\leq i\leq d_{s_0+\cdots+s_{p-1}+1}+d_{s_0+\cdots+s_{p-1}+2}+\cdots+d_{s_0+\cdots+s_{p-1}+r}-r$. Thus $d(v_{p,i})=d_{s_0+\cdots+s_{p-1}+i}$ for $1\leq i\leq q$. In this way, we obtain a tree $\mathcal{T}^*_{\pi}$ which is of order $n$ with degree sequence $\pi$.

\begin{lemma}\label{Lemma3.1}\cite{Zhang xiao dong3} For a given degree sequence $\pi$ of some tree, there exists a unique tree $\mathcal{T}^*_{\pi}$ with degree sequence $\pi$ having a $BFS$-ordering. Moreover, any two trees with same degree sequence and having $BFS$-ordering are isomorphic.
\end{lemma}

\vspace*{3mm}
\noindent{\bf Proof of Theorem \ref{Theorem3.3}.} Let $T$ be the tree that has largest $A_{\alpha}$-spectral radius in the class of all trees with degree sequence $\pi$. By Theorem \ref{Theorem2.7}, $T$ must have a BFS-ordering. Then $T\cong \mathcal{T}^*_{\pi}$ by Lemma \ref{Lemma3.1}.

\section{The $A_{\alpha}$-spectral radius of unicyclic graphs}
At first, we will introduce a special unicyclic graph $\mathcal{U}^*_{\pi}$ that has been defined by Zhang \cite{Zhang xiao dong2}. For a given nonincreasing degree sequence $\pi=(d_0,d_1,\ldots,d_{n-1})$ of a unicyclic graph with $n\geq3$, the $\mathcal{U}^*_{\pi}$ can be construct as follows: If $d_0=2$, then $\mathcal{U}^*_{\pi}=C_n$. If $d_0\geq3$ and $d_1=2$, then $\mathcal{U}^*_{\pi}$ consists of a triangle with $d_0-2$ hanging paths, attached at one vertex of the cycle, whose lengths are almost equal. If $d_1\geq3$, then we can use breadth-first-search method to defined $\mathcal{U}^*_{\pi}$ as follows. Select a vertex $v_{01}$ as a root and begin with $v_{01}$ of the zeroth layer. Put $s_1=d_0$ and select $s_1$ vertices $\{v_{11},v_{12},\ldots,v_{1s_1}\}$ of the first layer such that they are adjacent to $v_{01}$, and $v_{11}$ is adjacent to $v_{12}$. Thus $d(v_{01})=s_1=d_0$. Next we construct the second layer as follows. Select $s_2$ vertices $\{v_{21},v_{22},\ldots,v_{2s_2}\}$ of the second layer such that $d_{v_{11}}-2$ vertices adjacent to $v_{11}$, $d_{v_{12}}-2$ vertices adjacent to $v_{12}$ and $d_{v_{1i}}-1$ vertices adjacent to $v_{1i}$ for $i=3,\cdots,s_1$. In general, assume that all vertices of the $t$th layer have been constructed and are denoted by $\{v_{t1},v_{t2},\ldots,v_{ts_t}\}$. Now, using the induction hypothesis, we construct all vertices of the $t+1$th layer. Select $s_{t+1}$ vertices $\{v_{t+1,1},\ldots,v_{t+1,s_{t+1}}\}$ of the $t+1$th layer such that $d(v_{ti})-1$ vertices are adjacent to $v_{ti}$ for $i=1,\ldots,s_t$. In this way, we obtain only one unicyclic graph with degree sequence $\pi$.

\begin{lemma}\label{Lemma4.1}\cite{Zhang xiao dong2} For a given degree sequence $\pi$ of some unicyclic graph, $\mathcal{U}^*_{\pi}$ has a $BFS$-ordering.
\end{lemma}
\begin{lemma}\label{Lemma4.2}\cite{Zhang xiao dong2} Let $\pi=\{d_0,d_1,\ldots,d_{n-1}\}$ be a positive nonincreasing integer sequence with even sum and $n\geq3$. Then $\pi$ is a unicyclic graphic if and only if $\sum\limits_{i=0}^{n-1}d_i=2n$.
\end{lemma}
Let $G$ be a connected graph and $w$ be a vertex of $G$. Denote by $G(k,s)$ the graph obtain from $G\cup P_k\cup P_s$ by adding two edges between $w$ and end vertices of $P_k$ and $P_s$.
\begin{lemma}\label{Lemma4.3} \cite{Xue Jie1}Let $G(k,s)$ be the graph defined above with $k\geq s+2$. If $0\leq \alpha<1$ and $\rho(G(k,s))\geq2$, then
$$\rho(G(k,s))<\rho(G(k-1,s+1)).$$
\end{lemma}
\begin{lemma}\label{Lemma4.4} Let $\pi=(d_0,d_1,\ldots,d_{n-1})$ be a positive nonincreasing integer sequence with $d_1=2$ and $\sum\limits_{i=0}^{n-1}d_i=2n$. If $0\leq \alpha<1$, then $\rho(A_{\alpha}(\mathcal{U}^*_{\pi}))= \max\{\rho(A_{\alpha}(H))\mid \mbox{H is unicyclic and}~H\in \mathcal{G}_{\pi}\}$.
\end{lemma}
\begin{proof}Note that there must exist a unicyclic graph $H$ such that $H\in \mathcal{G}_{\pi}$ by Lemma \ref{Lemma4.2}. If $d_0=2$, then $H\cong C_n$. So we assume that $d_0\geq3$. Then $H$ must be the graph obtain from $C_p\cup\bigcup\limits_{i=1}^{d_0-2}P_{p_i}$ by adding $d_0-2$ edges between $v_0$ and the end vertices of $P_{p_i}$, $1\leq i\leq d_0-2$. Then the following claims must hold.

\noindent{\textbf{Claim 1.}} $p=3$. Otherwise, we can construct a new graph from $H$ as follows: The first step, constructing a graph $H_1$ by contracting an edge of $C_p$ in $H$, then $\rho(A_\alpha(H))<\rho(A_\alpha(H_1))$ by Lemma \ref{Lemma2.10}; The second step, constructing a graph $H_2$ by adding a pendant edge to a pendant vertex of $H_1$, then $\rho(A_\alpha(H_1))\leq\rho(A_\alpha(H_2))$ since $H_2\subset H_1$, thus $\rho(A_\alpha(H))<\rho(A_\alpha(H_2))$, which contradicts $H$ having the largest $A_\alpha$-spectral radius in $\mathcal{G}_{\pi}$ since $H_2\in \mathcal{G}_{\pi}.$

\noindent{\textbf{Claim 2.}} $|p_i-p_j|\leq1$. Otherwise, suppose that there exist two paths $P_{p_s}$ and $P_{p_t}$ such that $|p_s-p_t|\geq2$, without loss of generality, assume that $p_s\geq p_t+2$, thus $H$ can be expressed as $G(p_s,p_t)$. Note that $\rho(A_\alpha(H))>2$. Then by Lemma \ref{Lemma4.3}, we have $\rho(A_\alpha(G(p_s,p_t)))<\rho(A_\alpha(G(p_s-1,p_t+1)))$ and $G(p_s-1,p_t+1)\in \mathcal{G}_{\pi}$, which contradicts $H$ having the largest $A_\alpha$-spectral radius.

By the above Claims, we immediately obtain that $H$ must be $\mathcal{U}^*_{\pi}$.
\end{proof}
\begin{lemma}\label{Lemma4.5} \cite{V. Nikiforov2}Let $G$ be a graph with maximal degree $\Delta$. If $\alpha\in[0,1/2]$, then $$\rho(A_{\alpha}(G))\geq \alpha(\Delta+1).$$If $\alpha\in[1/2,1)$, then $$\rho(A_{\alpha}(G))\geq \alpha\Delta+1-\alpha.$$
\end{lemma}

\begin{lemma}\label{Lemma4.6} Let $\pi=(d_0,d_1,\ldots,d_{}n-1)$ be a positive nonincreasing integer sequence with $d_1\geq3$ and $\sum\limits_{i=0}^{n-1}d_i=2n$. If $0\leq \alpha<1$, then $\rho(A_{\alpha}(\mathcal{U}^*_{\pi}))= \max\{\rho(A_{\alpha}(H))\mid\mbox{H is unicyclic and}~H\in \mathcal{G}_{\pi}\}$.
\end{lemma}
\begin{proof}Note that there must exist a unicyclic graph $H$ such that $H\in \mathcal{G}_{\pi}$ by Lemma \ref{Lemma4.2}. Let $G$ be the graph with the largest $A_\alpha$-spectral radius in $\mathcal{G}_{\pi}$ and $X$ be a unit eigenvector of $A_{\alpha}(G)$ corresponding to $\rho(A_{\alpha}(G))$ in which $x_i$ is label of $u_i$. Then by Theorem \ref{Theorem2.7}, $G$ has a BFS-ordering, i.e., there exists a well-ordering of the vertices of $G$ such that
$$v_0\prec v_1\prec\cdots\prec v_{n-1},$$
$$x_{v_0}\geq x_{v_1}\geq \cdots\geq x_{v_{n-1}},$$
$$d(v_0)\geq d(v_1)\geq \cdots\geq d(v_{n-1}),$$
and
$$h(v_0)\leq h(v_1)\leq \cdots\leq h(v_{n-1}).$$
Let $V_i=\{v\mid v\in V(G), h(v)=i\}$ for $i=0,\cdots,p=h(v_{n-1})$. Then we can relabel the vertices of $G$ in such a way that $V_i=\{v_{i1}, \ldots,v_{is_i}\}$ with $x_{v_{i1}}\geq x_{v_{i2}}\geq\cdots\geq x_{v_{is_i}}$, $x_{v_{ij}}\geq x_{v_{i+1,k}}$ and $d(v_{ij})\geq d(v_{i+1,k})$ for $0\leq i\leq p-1$, $1\leq j\leq s_i$ and $1\leq k\leq s_{i+1}$. Clearly, $s_1=d(v_0)=d_0$. And the following claim must hold.

\noindent{\textbf{Claim.}} $x_{v_{01}}>x_{v_{1s_1}}$. Otherwise, let $x_{v_{01}}=x_{v_{1s_1}}$, then $x_{v_{01}}=x_{v_{11}}=\cdots=x_{v_{1s_1}}$. By $A_\alpha(G)X=\rho(A_\alpha(G))X$, we have $$\rho(A_\alpha(G))x_{v_{01}}=d_0x_{v_{01}},$$ which implies that $\rho(A_\alpha(G))=d_0$ and $G$ is a regular graph by Lemma \ref{Lemma2.9}, a contradiction.

Let $C$ be the unique cycle of $G$ and $v_{rq}$ be the vertex with smallest height among vertices in $V(C)$, that is, for any vertex $u\in V(C)$, we have
\begin{center}
$h(v_{rq})=r\leq h(u)$ and $d(v_{rq})\geq d(u).$
\end{center}
Then we can discuss by the following five cases.

\noindent{\textbf{Case 1.}} $v_{rq}=v_{01}$ and $v_{11}v_{12}\in E(G)$. Then $C=v_{01}v_{11}v_{12}v_{01}$ and $G$ is $\mathcal{U}^*_{\pi}$ obviously.

\noindent{\textbf{Case 2.}} $v_{rq}=v_{01}$, $v_{11}v_{12}\notin E(G)$ and $v_{1i}v_{1j}\in E(G)$ for $1\leq i<j\leq s_1$.

\noindent{\textbf{Subcase 2.1.}} $i=1$ and $3\leq j\leq s_1$. Without loss of generality, assume that $x_{v_{12}}>x_{v_{1j}}$ (otherwise, we can exchange $v_{12}$ and $v_{1j}$, then it is the same as case 1.), then $d(v_{12})\geq d(v_{1j})\geq2$ by Lemma \ref{Lemma2.5}. Note that there exists a vertex $v_{2t}\in V_2$ such that $v_{12}v_{2t}\in E(G)$ and $v_{1j}v_{2t}\notin E(G)$ since $G$ is unicyclic graph. $G'_1\in \mathcal{G}_{\pi}$ is a new graph that is obtained from $G$ by deleting edges $v_{11}v_{1j}$ and $v_{12}v_{2t}$ and adding edges $v_{11}v_{12}$ and $v_{1j}v_{2t}$. Since $x_{v_{11}}>x_{v_{2t}}$ and $x_{v_{12}}>x_{v_{1j}}$, $\rho(A_\alpha(G))<\rho(A_\alpha(G'_1))$ by Lemma \ref{Lemma2.3}, a contradiction.

\noindent{\textbf{Subcase 2.2.}} $1<i<j\leq s_1$. Without loss of generality, assume that $x_{v_{11}}>x_{v_{1i}}$ (otherwise, we can exchange $v_{11}$ and $v_{1i}$, then it is as same as subcase 2.1.), then $d(v_{11})\geq d(v_{1i})\geq2$ by Lemma \ref{Lemma2.5}. Note that there exists a vertex $v_{2t}\in V_2$ such that $v_{11}v_{2t}\in E(G)$ and $v_{1j}v_{2t}\notin E(G)$ since $G$ is unicyclic graph. $G'_2\in \mathcal{G}_{\pi}$ is a new graph that is obtained from $G$ by deleting edges $v_{11}v_{2t}$ and $v_{1i}v_{1j}$ and adding edges $v_{11}v_{1i}$ and $v_{1j}v_{2t}$. Since $x_{v_{11}}>x_{v_{1i}}$ and $x_{v_{1j}}\geq x_{v_{2t}}$, $\rho(A_\alpha(G))<\rho(A_\alpha(G'_2))$ by Lemma \ref{Lemma2.3}, a contradiction.

\noindent{\textbf{Case 3.}} $v_{rq}=v_{01}$ and $v_{1i}v_{1j}\notin E(G)$ for $1\leq i<j\leq s_1$. Then there must exist $v_{1k}\in V_1$ and $v_{2l}\in V_2$ such that $v_{1k}v_{2l}\in E(C)$ and $k\geq2$. Since $d(v_{11})=d_1\geq3$, there exists a vertex $v_{2t}\in V_2$ such that $v_{11}v_{2t}\in E(G)\backslash E(C)$ and $x_{v_{2t}}\leq x_u$, where $u\in N_{v_{11}}$ and $v\notin V(C)$. Then $v_{1k}v_{2t}\notin E(G)$ and $v_{2l}v_{2t}\notin E(G)$ since $G$ is unicyclic graph. The following claim about $G$ holds.

\noindent{\textbf{Claim.}} $x_{v_{11}}>x_{2l}$ or $x_{v_{1k}}>x_{v_{2t}}$. Otherwise, let $x_{v_{11}}=x_{2l}$ and $x_{v_{1k}}=x_{v_{2t}}$, then $x_{v_{11}}=x_{2l}=x_{v_{1k}}=x_{v_{2t}}$, $x_u=x_{v_{11}}$ for $u\in N_{v_{11}}\backslash \{v_{01}\}$ and $x_{v_{11}}=\cdots=x_{v_{1s_1}}$. By $A_\alpha(G)X=\rho(A_\alpha(G))X$, we get that
\begin{eqnarray*}
\rho(A_\alpha(G))x_{v_{2t}}&=&\alpha d(v_{2t})x_{v_{2t}}+(1-\alpha)\sum\limits_{uv_{2t}\in E(G)}x_u\\
&\leq& \alpha d(v_{2t})x_{v_{11}}+(1-\alpha)d(v_{2t})x_{v_{11}}\\
&=& d(v_{2t})x_{v_{11}}
\end{eqnarray*}
that is, $\rho(A_\alpha(G))\leq d(v_{2t})=d(v_{11})=d_1.$
\begin{eqnarray*}
\rho(A_\alpha(G))x_{v_{01}}&=&\alpha d(v_{01})x_{v_{01}}+(1-\alpha) d(v_{01})x_{v_{11}}\\
&=& d_0(\alpha x_{v_{01}}+(1-\alpha)x_{v_{11}})
\end{eqnarray*}
that is, $(\rho(A_\alpha(G))-\alpha d_0)x_{v_{01}}=(1-\alpha)d_0x_{v_{11}}.$
\begin{eqnarray*}
\rho(A_\alpha(G))x_{v_{11}}&=&\alpha d(v_{11})x_{v_{11}}+(1-\alpha)\sum\limits_{uv_{11}\in E(G)}x_u\\
&=& \alpha d(v_{11})x_{v_{11}}+(1-\alpha)d(v_{11-1})x_{v_{11}}+(1-\alpha)x_{v_{01}}\\
&=& d_1x_{v_{11}}+(1-\alpha)(x_{v_{01}}-x_{v_{11}})
\end{eqnarray*}
that is, $(\rho(A_\alpha(G))-d_1-\alpha+1)x_{v_{11}}=(1-\alpha)x_{v_{01}}.$
By the above equalities, we have
$$(\rho(A_\alpha(G))-d_1)[\rho(A_\alpha(G))-\frac{d_0(\alpha d_1-(1-\alpha)(1+2\alpha))}{d_1}]=\rho(A_\alpha(G))(\alpha-1)\frac{d_1-d_0-2\alpha d_0}{d_1}.$$
Note that $\rho(A_\alpha(G))>\frac{d_0(\alpha d_1-(1-\alpha)(1+2\alpha))}{d_1}$ by Lemma \ref{Lemma4.5}. Combining with $\rho(A_\alpha(G))\leq d_1$, we have $\rho(A_\alpha(G))=d_1=(1+2\alpha)d_0>d_0$, which contradicts $\rho(A_\alpha(G))\leq d_1\leq d_0$. Thus,  the claim must hold.

Now, we can construct a new graph $G'_3\in \mathcal{G}_{\pi}$ that is obtained from $G$ by deleting edges $v_{11}v_{2t}$ and $v_{1k}v_{2l}$ and adding edges $v_{11}v_{1k}$ and $v_{2l}v_{2t}$. Combining with the above claim, we have $\rho(A_\alpha(G))<\rho(A_\alpha(G'_3))$ by Lemma \ref{Lemma2.3}, a contradiction.

\noindent{\textbf{Case 4.}} $v_{rq}=v_{11}$. Then $v_{11}v_{1i}\notin E(G)$ for any $2\leq i\leq s_1$ and there must exist $v_{2j}\in V(C)$ such that $v_{11}v_{2j}\in E(G)$, then $d(v_{2j})\geq 2$. And then $d(v_{12})\geq2$. So there exists a vertex $v_{2t}\in V_2$ such $v_{12}v_{2t}\in E(G)$ since $G$ is unicyclic graph. By the same argument as the claim of case 3, it can be shown that $x_{v_{11}}>x_{v_{2t}}$ or $x_{v_{12}}>x_{v_{2j}}$. Then we can construct a new graph $G'_4\in \mathcal{G}_{\pi}$ that is obtained from $G$ by deleting edges $v_{11}v_{2j}$ and $v_{12}v_{2t}$ and adding edges $v_{11}v_{12}$ and $v_{2j}v_{2t}$. By Lemma \ref{Lemma2.3}, we have $\rho(A_\alpha(G))<\rho(A_\alpha(G'_3))$, a contradiction.

\noindent{\textbf{Case 5.}} $v_{rq}\neq v_{01},v_{11}$.

\noindent{\textbf{Claim.}} $h(v_{rq})<h(u)$ for any $u\in V(C)\setminus \{v_{rq}\}$. Otherwise, suppose that $v_{rl}\in V(C)$ and $v_{rl}\neq v_{rq}$ such that $h(v_{rq})=h(v_{rl})$, then there exist two internal disjoint paths $P_{l_1}$ and $P_{l_2}$ between $v_{rq}$ and $v_{rl}$ such that $C=P_{l_1}\cup P_{l_2}$. However, there exist two paths $P_{l_3}$ and $P_{l_4}$ from $v_{01}$ to $v_{rq}$ and $v_{rl}$, respectively. And $P_{l_3}\cup P_{l_4}\cup P_{l_1}$ contains another cycle of $G$, which contradicts $G$ is unicyclic graph.

Then there is a vertex $v_{r+1,t}\in V_{r+1}$ such that $v_{rq}v_{r+1,t}\in E(C)$, thus $d(v_{ki})\geq d(v_{r+1,t})\geq2$ for any $k\leq r$ and $1\leq i\leq s_k$. And we can find two vertices $v_{ri}$ and $v_{r+1,j}$ such that $v_{ri}v_{r+1,j}\in E(G)\setminus E(C)$. Obviously, $v_{rq}v_{r+1,j}, v_{r+1,t}v_{r+1,j}\notin E(G)$. $G'_4\in \mathcal{G}_{\pi}$ is a new graph that is obtained from $G$ by deleting edges $v_{rq}v_{r+1,t}$ and $v_{ri}v_{r+1,j}$ and adding edges $v_{rq}v_{ri}$ and $v_{r+1,t}v_{r+1,j}$. Combining with $x_{v_{rq}}\geq x_{v_{r+1,j}}$ and $x_{v_{ri}}\geq x_{v_{r+1,t}}$, we have $\rho(A_\alpha(G))\leq\rho(A_\alpha(G'_5))$ by Lemma \ref{Lemma2.3}. Furthermore, the smallest height of vertices of the cycle in $G_5$ is less than $r$. By repeating the use of Case 5 or Cases 2,3 and 4, we can get $\rho(A_\alpha(G))\leq\rho(A_\alpha(G'_5))<\rho(A_{\alpha}(\mathcal{U}^*_{\pi}))$, a contradiction.
\end{proof}

By Lemmas \ref{Lemma4.4} and \ref{Lemma4.6}, Theorem \ref{Theorem4.7} can be obtained immediately.

\end{document}